\documentclass[10pt]{amsart}
\usepackage{amssymb, amsmath,xcolor}
\usepackage{soul}
\newtheorem{theorem}{Theorem}

\newtheorem{remark}{Remark} 
\newtheorem{lemma}{Lemma} 

\DeclareMathOperator{\HF}{HF}
\DeclareMathOperator{\HH}{H}
\DeclareMathOperator{\ini}{in}

\title[Monomial ideals and the failure of the SLP]{
Monomial ideals and the failure of the Strong Lefschetz property
}

\author{Nasrin Altafi}\address[Corresponding author]{Department of Mathematics, KTH Royal Institute of Technology, Sweden}\email{nasrinar@kth.se} 

\author{Samuel Lundqvist}\address{Department of Mathematics, Stockholm University, Sweden}\email{samuel@math.su.se}


\begin{document}

\begin{abstract}
We give a sharp lower bound for the Hilbert function in degree $d$ of artinian quotients $\Bbbk[x_1,\ldots,x_n]/I$ failing the Strong 
Lefschetz property, where $I$ is a monomial ideal generated in degree $d \geq 2$.
We also provide sharp lower bounds for other classes of ideals, and connect our result to the classification of the Hilbert functions forcing the Strong Lefschetz property by Zanello and Zylinski.
\end{abstract}
\maketitle

\noindent

\section{Introduction}
An artinian graded algebra has the Weak Lefschetz property 
(WLP) 
if there is a 
linear form $\ell$ such that the map induced by multiplication by $\ell$ is either surjective or injective in every degree, and the element 
$\ell$ is then called a Weak Lefschetz element, while an artinian graded algebra has the Strong Lefschetz property (SLP) 
if there is a linear form $\ell$ such that the map induced by multiplication by $\ell^i$ is either surjective or injective  in every degree, 
for all $i\geq 1$, and the element $\ell$ is then called a Strong Lefschetz element.

It has long been known that every monomial complete intersection in characteristic zero has the SLP. This result is due to Stanley  \cite{stanley} and independently by Watanabe \cite{Watanabe}. One way of thinking about this result is that an artinian algebra of the form 
$\Bbbk[x_1,\ldots,x_n]/I$, where $I$ is a monomial ideal, has the SLP when the number of minimal generators is as small as possible (equal to $n$). 
When we increase the number of monomial generators of the ideal, it is not always the case that the quotient algebra has the SLP, or even the WLP. 
Indeed, for each $n\geq 3$, the equigenerated monomial almost complete intersection 
$\Bbbk[x_1,\ldots,x_n]/(x_1^n,\ldots,x_n^n,x_1 \cdots x_n)$ fails the WLP. This was noticed for $n = 3$ by Brenner and Kaid \cite{brenner}, while the 
proof of the general case $n \geq 4$ is attributed to Migliore, Mir\'o-Roig, and Nagel \cite{MMN}.

Given these results, it is natural to consider the following question: Suppose that $I$ is an equigenerated artinian monomial ideal such that $\Bbbk[x_1,\ldots,x_n]/I$ 
fails one of the Lefschetz properties. What can we say about the number of minimal generators of $I$?

The first result in this direction goes back to Mezzetti and Mir\'o-Roig \cite{MM} who 
provided a sharp upper bound for the Hilbert function in degree $d$ of minimal monomial Togliatti systems. Recall that a monomial ideal 
$I\subset \Bbbk[x_1,\ldots,x_n]$ generated in degree $d$ is called a monomial Togliatti system if the quotient $\Bbbk[x_1,\ldots,x_n]/I$ fails the WLP in 
degree $d-1$ by failing injectivity. 
Togliatti systems were defined in  \cite{MMO} and the name is in honour of Engenio Togliatti who characterized smooth Togliatti systems for 
$n=d=3$ \cite{Togliatti1,Togliatti2}.

The first author and Boij \cite{nasrinmats} recently 
gave the following sharp lower bound for the Hilbert function in degree $d$ of an artinian quotient 
$\Bbbk[x_1,\ldots,x_n]/I$ failing the WLP, where $I$ is a monomial ideal generated in degree $d$.

\begin{theorem} \cite[Theorem 1.1, Theorem 1.2]{nasrinmats} \label{thm1}
Let $n\geq 3$ and let $\Bbbk$ be a field of characteristic zero. Let $I\subset S=\Bbbk[x_1, \ldots ,x_n]$ be a monomial ideal generated in degree $d\geq 2$. 
Assume that $R=S/I$ fails the WLP. Then 
$$\HF(R,d) \geq  \begin{cases}
3(d-1) & \text{if } n=3 \text{ and } d \text{ odd,} \\
3(d-1)+1 & \text{if } n=3 \text{ and } d \text{ even,} \\
2d & \text{ if } n\geq 4,
\end{cases}$$
where $\HF(R,d)$ is the value of the Hilbert function in degree $d$ of $R$.

Furthermore, the bounds are sharp.

\end{theorem}
In this note, we carry out a similar program for the SLP, resulting in the following classification.
\begin{theorem} \label{thm2}
Let $n \geq 3$ and let $\Bbbk$ be a field of characteristic zero. Let $I \subset S= \Bbbk[x_1,\ldots,x_n]$ be an artinian monomial ideal 
generated in degree $d \geq 2$. Let $R=S/I$ and denote the value of the Hilbert function in degree $d$ of $R$ by $\HF(R,d)$.
\begin{itemize}
\item[$(i)$]Assume that $R$ fails the SLP.  Then
$$\HF(R, d) \geq 
\begin{cases}
4 & \text{if } d= 2, \\
3 & \text{if } d \geq 3.
\end{cases}$$
Moreover, this is sharp in the sense that for any $d$ there is an example for which  $\ell^{d-1}:R_1\rightarrow R_d$ fails to be surjective and the value of the Hilbert function in degree $d$ is given by this bound.
\item[$(ii)$] Let $1\leq i\leq d-1$. If $\ell^i$ fails to have maximal rank in some degree then 
$$ \HF(R,d)\geq d-i+2.$$
Moreover, the bound is sharp for every $2\leq i\leq d-1$: there is an example where $\ell^{i}:R_{d-i}\rightarrow R_d$ is not surjective and the bound is achieved.
\end{itemize}

\end{theorem}

\begin{remark} Over a field of characteristic zero, 
Harima, Migliore, Nagel, and Watanabe showed that every artinian homogeneous quotient 
in two variables has the SLP \cite{uwe}, explaining why the case $n = 2$ is omitted in the results.
\end{remark}

\begin{remark}
In the case $n=3, d= 2,$ a computation reveals that all quotients of artinian monomial ideals enjoy the WLP and the SLP, and this is the only case for which the conditions in Theorem \ref{thm1} 
and Theorem \ref{thm2} are empty.
\end{remark}

\begin{remark}

In Theorem \ref{thm2} part (ii), the bound holds for every  $1\leq i\leq d-1$ although when $i=1$ multiplication by $\ell$ does not have full rank which means the WLP fails. In this case, the bound in Theorem \ref{thm2} part (ii) is not sharp.
\end{remark}

The consequences of Theorem \ref{thm2} are stronger than one would first expect, and for $d \geq 3$, we will show that Theorem \ref{thm2} holds for other natural classes of artinian ideals, including ideals generated by forms of degree $d$. We will see that this generalization gives a bridge to a classification result on the Hilbert functions that force all ideals to have the SLP by Zanello and Zylinski \cite{Zanello}.

\section{Preliminaries}
Throughout the remaining part of this paper, let $\Bbbk$ be a field of characteristic zero and let $S = \Bbbk[x_1,\ldots,x_n]$, where $n \geq 3$. We will use that a monomial algebra $R$ has the SLP if and only if $\ell := x_1 + \cdots + x_n$ is a Strong Lefschetz element \cite[Remark 4.5] {atour}. 

Duality arguments will be crucial for our results and we recall here the basic notion of the inverse system. Let $\mathcal{E}=\Bbbk[y_1, \dots ,y_n]$ be the Macaulay dual ring to 
$S$, where $S$ acts on $\mathcal{E}$ by differentiation; i.e. $x_j\circ f=\partial f/\partial y_j$, for every homogeneous polynomial $f\in \mathcal{E}$. There is a 
bijection between the set of finitely generated $S$-submodules $M$ of $\mathcal{E}$ and the set of artinian quotients $S/I$ given by $I=\mathrm{Ann}_S(M)=(0:_S M)$, 
and $I^{-1}=M$. The module $I^{-1}$ is called the inverse system module of $I$, for more details see \cite{Geramita} and \cite{IK}. Using this correspondence, we 
have that $\HF(R,i)=\dim_\Bbbk(I^{-1})_i$. For every $i$, the multiplication map $\ell:(S/I)_i\rightarrow (S/I)_{i+1}$ has maximal rank if and only if the differentiation 
map  $\circ \ell:(I^{-1})_{i+1}\rightarrow (I^{-1})_i$ has maximal rank. We observe that when $I$ is a monomial ideal, for every $i$  the $i$-th graded piece of the  
inverse system module, $(I^{-1})_i$, is generated by monomials in $\mathcal{E}_d$ that are dual to the monomials in $S_i\setminus I_i$.

Finally, we state one more result from \cite{nasrinmats} which will be central to our argument.

\begin{theorem}  \cite[Theorem 3.7]{nasrinmats} \label{thm37}
Let $f\in \mathcal{E}=\Bbbk[y_1,\dots ,y_n]$ be a polynomial of degree $d$ such that $\ell^i\circ f=0$, for some $1\leq i\leq d$. Then the number of monomials with 
non-zero coefficient in $f$, $\vert \mathrm{supp}(f)\vert$, is at least $d-i+2$.
\end{theorem}

\section{ Main result}\par 

Before providing the proof of the main result of the paper we state the following two lemmas in which we provide equivalent conditions for the map $\ell^i$ to have maximal rank for $1\leq i\leq d-1$. Notice that we do not require the ideals in Lemma \ref{lemma-l^i} and Lemma \ref{SLPlemma} to be equigenerated.
\begin{lemma}\label{lemma-l^i}
Let $R=S/I$ be an artinian algebra, where $n\geq 3$ and 
 $I = (m_1,\ldots,m_s)$ a monomial ideal such that $\min \{\deg(m_i)\} = d \geq 2$. Assume that $\HF(R, d-i)\geq \HF(R,d)$ for some 
$1\leq i\leq d-1$. Then the map $\ell^{i}:R_{d-i}\rightarrow R_d$ is surjective if and only if  the map $\ell^i:R_j\rightarrow R_{i+j}$ has maximal rank for every $j$.
\end{lemma}
\begin{proof}
Suppose that $\ell^i:R_{d-i}\rightarrow R_d$ is surjective. This means $\left[R/\ell^i R\right]_d=0$ and also $\left[R/\ell^i R\right]_{d+k}=0$ for every $k\geq 1 $ which 
implies that $\ell^i:R_{d+k-i}\rightarrow R_{d+k}$ is surjective, for every $k\geq 1$. On the other hand, since degrees of the generators of $I$ is at least $d$, for every $k\geq 1$ the 
map $\ell^i:R_{d-k-i}\rightarrow R_{d-k}$ is trivially injective. The other assertion follows immediately by setting $j=d-i$.
\end{proof}
\begin{lemma}\label{SLPlemma}
Let $R=S/I$ be an artinian algebra, where $n\geq 3$ and 
 $I = (m_1,\ldots,m_s)$ a monomial ideal such that $\min \{\deg(m_i)\} = d \geq 2$.
Suppose that $\HF(R, 1)\geq \HF(R,d)$. Then $R$ has 
the SLP if and only if the map $\ell^{d-1}:R_1\rightarrow R_d$ is surjective.
\end{lemma}
\begin{proof}
Assuming that $\HF(R, 1)\geq \HF(R,d)$ implies that $\HF(R, d-i)\geq \HF(R,d)$, for every $1\leq i\leq d-1$.  Now suppose that $\ell^{d-1}:R_1\rightarrow R_d$ is surjective. Then
the map $\ell^i: R_{d-i}\rightarrow R_d$ is also surjective for every $1\leq i\leq d-1$. Lemma \ref{lemma-l^i} implies that for every $1\leq i\leq d-1$ the map $\ell^i$ has 
maximal rank in all degrees. On the other hand, since the map $\ell$ has maximal rank in all degrees, the Hilbert function of $R$ is unimodal, see \cite{uwe}. Therefore, for every $k\geq 0$ and 
$1\leq i\leq d-1$ we have $\HF(R,i)\geq \HF(R,d+k+i)$, so the map $\ell^{d+k}:R_i\rightarrow R_{d+k+i}$  has maximal rank if and only if it is surjective. We notice that $\ell^{d+k}:R_i\rightarrow R_{d+k+i}$ is the composition of two surjective maps  $\ell^{d-i}:R_i\rightarrow R_{d}$ and $\ell^{i+k}:R_d\rightarrow R_{d+k+i}$ and therefore is surjective. We have shown that $R$ has the SLP. The other 
implication is trivial.
\end{proof}

Now we are able to prove our result for equigenerated monomial ideals failing the SLP.
\begin{proof}[Proof of Theorem \ref{thm2}] We first prove the second part.
\begin{itemize}
\item[$(ii)$]
Fix an integer $i$ such that $1\leq i\leq d-1$. First assume that  $\HF(R, d-i)\geq \HF(R,d)$. Suppose that  $\ell^i:R_j\rightarrow R_{i+j}$ does not have maximal rank for some $j\geq 1$. 
Using Lemma \ref{lemma-l^i}, we get that the map $\ell^i:R_{d-i}\rightarrow R_{d}$ is not surjective. Therefore,  Theorem \ref{thm37} implies that $\HF(R,d)\geq d-i+2$. 
Now assume that  
$\HF(R, d)>\HF(R,d-i)$. For $1\leq i\leq d-1$ we have
$$\HF(R, d) >\HF(R,d-i)=\binom{n+d-i-1}{n-1}\geq d-i+2,$$
where the last inequality follows inductively from
$$\binom{n+d-i-1}{n-1} = \frac{n+d-i-1}{n-1} \binom{(n-1)+d-i-1}{(n-1)-1}$$ and
$$\binom{3+d-i-1}{3-1} = \binom{d-i+2}{2} \geq d-i+2. $$

In order to show that the bound is sharp for every $2\leq i\leq d-1$, consider $f = y^{i-1}_1(y_2-y_3)^{d-i+1}\in \mathcal{E}_d$ and let $I\subset S$ to be the monomial 
ideal generated by monomials in $S_d$ dual to the monomials in $\mathcal{E}_d\setminus \mathrm{Supp}(f)$. We have that $f\in (I^{-1})_d$ and 
$$\HF(R, d) = \HF(S, d)-\vert \mathcal{G}(I)\vert =\vert \mathrm{supp}(f)\vert = d-i+2,$$ where $\mathcal{G}(I)$ is the minimal generating set of $I$. Observe that  
$\ell^i\circ f = 0$ which implies that the differentiation map  $\circ \ell^i: (I^{-1})_d\rightarrow (I^{-1})_{d-i}$ is not injective, or equivalently, that  $\ell^i:R_{d-i}\rightarrow R_{d}$ 
is not surjective.

\item[$(i)$]
Let $d\geq 3$ and assume that $R$ fails the SLP.

First suppose that $\HF(R, 1)\geq \HF(R,d)$. Then using Lemma \ref{SLPlemma} we conclude that the map $\ell^{d-1}:R_1\rightarrow R_d$ is not surjective and Theorem \ref{thm37} 
implies that $\HF(R,d)\geq 3$. Now suppose we have that  $\HF(R, d)> \HF(R,1)=n\geq 3$, so we get the desired inequality. In this case, $d\geq 3$, the sharpness of the bound is implied 
by the second part. In fact, the polynomial $f=y^{d-2}_1(y_2-y_3)^2$ is in the kernel of the map $\circ \ell^{d-1}:(I^{-1})_d\rightarrow (I^{-1})_1$ and we have that 
$\HF(R, d) =\vert \mathrm{supp}(f)\vert = 3$.

Now assume that $d=2$ and that $R$ fails the SLP. 

If $n=3$ then a calculation shows that every artinian quadratic monomial ideal in three variables has the SLP, so there is nothing to prove. Thus we assume $n\geq 4$. By Theorem \ref{thm1}, the map $\ell: R_1 \to R_2$ is surjective when $\HF(R,2) < 4$, so by Lemma \ref{SLPlemma}, $R$ has the SLP.
Therefore, if $R$ fails the SLP then $\HF(R,2)\geq 4$ and the sharpness of the bound is implied by Theorem \ref{thm1}.
\end{itemize}
\end{proof}

We turn directly to the generalization of Theorem \ref{thm2}. 

\begin{theorem} \label{cor}

Let $n \geq 3$ and let $\Bbbk$ be a field of characteristic zero. Let $I \subset S= \Bbbk[x_1,\ldots,x_n]$ be an artinian  ideal 
generated in degree $d \geq 3$, or an artinian monomial ideal 
 $(m_1,\ldots,m_s)$  such that $\min \{\deg(m_i)\} = d \geq 3$, or an artinian ideal 
 $(f_1,\ldots,f_s)$  such that $\min \{\deg(f_i)\} = d \geq 3$. 
Let $R=S/I$ and denote the value of the Hilbert function in degree $d$ of $R$ by $\HF(R,d)$.
\begin{itemize}
\item[$(i)$]Assume that $R$ fails the SLP.  Then 
$$\HF(R, d) \geq 3.$$ Moreover, this is sharp in the sense that for any $d$ there is an example belonging to the specific class for which  $\ell^{d-1}:R_1\rightarrow R_d$ fails to be surjective.
\item[$(ii)$] Let $1\leq i\leq d-1$. If $\ell^i$ fails to have maximal rank in some degree then 
$$ \HF(R,d)\geq d-i+2.$$
Moreover, the bound is sharp for every $2\leq i\leq d-1$: there is an example where $\ell^{i}:R_{d-i}\rightarrow R_d$ is not surjective and the bound is achieved.
\end{itemize}

\end{theorem}

\begin{proof}

Let $R^\prime=S/J$ where $J= \ini(I)$ is the initial ideal of $I$ with respect to a term order. Notice that $J$ is a monomial ideal such that $\min\{\deg(m_i)\}=d\geq 3$ and therefore Lemmas \ref{lemma-l^i} and \ref{SLPlemma} hold for $J$. 

Wiebe's result \cite[Proposition 2.9]{Wiebe} states that if $R^\prime$ has the SLP then the same holds for $R$. The proof  of \cite[Proposition 2.9]{Wiebe} actually reveals a more general fact: for a general enough $\tilde{\ell}$ and for  every $i,j\geq 0$, $\tilde{\ell}^j:R_i\rightarrow R_{i+j}$ has maximal rank if $\ell^j:R^\prime_i\rightarrow R^\prime_{i+j}$ has maximal rank. 

The inequalities now follow by the same line as of the proof of Theorem \ref{thm2}.  For the sharpness parts, we use that artinian monomial ideals generated in degree $d$ are part of each of the three classes.

\end{proof}

\begin{remark}
For $d=2$, the bounds in Theorem 1 and Theorem 2 coincide, and in fact, in the proof of Theorem \ref{thm2}, we use Theorem \ref{thm1} to obtain sharp bounds for $\HF(R,2)$ for artinian monomial ideals generated in degree $2$. But for the classes of artinian ideals in Theorem \ref{cor}, the same approach does not provide a bound for $\HF(R,2)$, since Theorem 1 is only valid for equigenerated monomial ideals.

However, when $n=3$, if a quadratic artinian ideal $I$ fails the SLP,  then $\HF(R,2)\geq 4$, which coincides with the bound given in Theorem \ref{thm2}. To see this, notice that if $\HF(R,2)\leq 2$, then Theorem \ref{thm37} implies that $\ell: [{S/\ini(I)}]_1\rightarrow {[S/\ini(I)]}_2$ is surjective and therefore, using Lemma \ref{SLPlemma}, we have that $S/\ini(I)$, and thus $R$, has the SLP. If $\HF(R,2)=3$, then $I$ is a complete intersection having Hilbert function $(1,3,3,1)$, which has the SLP according to the Gordan-Noether theorem \cite{GN}.

\end{remark}

\section{A connection to the Migliore-Zanello and the Zanello-Zylinski classification results}\par

Migliore and Zanello \cite{MZ} classified the Hilbert functions that force the WLP, and later Zanello and Zylinski  \cite{Zanello} classified the Hilbert functions that force the SLP.

\begin{theorem} \label{Zanello}
Let $\HH : 1,h_1=n, h_2,\dots ,h_e,h_{e+1}=0$, be a possible Hilbert function, according to Macaulay's theorem \cite{Mac}, and let $t$ be the smallest integer such that $h_t\leq t$. Then we have the following. 
\begin{itemize}
\item[$(i)$] \cite[Theorem 5]{MZ} All artinian algebras having the Hilbert function $\HH$ enjoy the WLP if and only if, for all $i=1,2,\dots , t-1$, we have 
$$
h_{i-1}=((h_i)_{(i)})_{-1}^{-1}.
$$
\item[$(ii)$]  \cite[Theorem 3.2]{Zanello} All artinian algebras having the Hilbert function $\HH$ enjoy the SLP if and only if
\begin{itemize}
\item $n=2$; or
\item $n>2, h_t\leq 2$, and, for all $i=1,2,\dots ,t-1$, we have 
$$
h_{i-1}=((h_i)_{(i)})_{-1}^{-1}.
$$
\end{itemize}
\end{itemize}
\end{theorem}
In the theorem above, for integers $m$ and $i$, the $i$-binomial expansion of $m$ is denoted by $m_{(i)}$, that is, 
$$
m=m_{(i)}=\binom{m_i}{i}+\binom{m_{i-1}}{i-1}+\cdots + \binom{m_{j}}{j},
$$
where $m_i\geq m_{i-1}\geq \dots \geq j \geq 1$. Also
$$
(m_{(i)})_{-1}^{-1}=\binom{m_i-1}{i-1}+\binom{m_{i-1}-1}{i-1-1}+\cdots + \binom{m_{j}-1}{j-1},
$$
where $\binom{c}{d}=0$ whenever $c<d$ or $d<0$.

We will now discuss how Theorem \ref{Zanello} is related to the results in Theorem \ref{thm1}, Theorem \ref{thm2}, and Theorem \ref{cor}.

We consider first the WLP and the case $n \geq 4$. Suppose $I$ is a monomial ideal equigenerated in degree $d$ and $\HF(R,d) \leq d$. Both Theorem \ref{thm1}  and Theorem \ref{Zanello} part (i) imply that $R$ has the WLP. Now let $\HF(R,d)=i$ and suppose that $d+1 \leq i \leq 2d-1$. Theorem \ref{thm1} implies that $R$ has the WLP. On the other hand, $$i_{(d)} = \binom{d+1}{d} + \binom{d-1}{d-1}+ \cdots +  \binom{d-(i-(d+1))}{d-(i-(d+1))},$$ so $(i_{(d)})_{-1}^{-1} = i-1\leq 2d-2\neq \HF(R,d-1)$. Therefore, the requirement of Theorem \ref{Zanello} part (i) is not satisfied. This shows that  there is an artinian ideal such that $\HF(R,d) = i $  which fails the WLP. Moreover, by considering the initial ideal of $I$, Wiebe's result implies that there is an artinian monomial ideal such that $\HF(R,d) = i $  which fails the WLP. A similar argument can be applied for the $n=3$ case.

We now turn to the SLP. If $I$ is generated in degrees $d\geq 2$ and $\HF(R,d)\leq 2$ then one can easily check that $t=d$ in Theorem \ref{Zanello} part (ii), so $R$ has the SLP. This gives an alternative proof of one direction in part (i) of Theorem  \ref{cor}. 

In the quadratic case however, for each $n \geq 3$  there is an artinian ideal such that $\HF(R,2) = 3$ and for which the SLP fails for $R$  by Theorem \ref{Zanello} part (ii), although for all monomial artinian equigenerated ideals such that  $\HF(R,d) = 3$, the SLP holds according to Theorem \ref{thm2}. Moreover, by considering the initial ideal, by Wiebe's result we conclude that there is an artinian monomial ideal such that $\HF(R,d) =3$  which fails the SLP.

\subsection*{Acknowledgements}
The initial investigation was performed with the help of \emph{Macaulay2} \cite{M2} and the package Maximal rank properties \cite{maximalrank}. The authors thank the anonymous referee for useful comments that improved  both the results and the presentation of the paper.

\end{document}